\documentclass[12pt]{amsart}
\usepackage{SSdefn}
\usepackage[lite,alphabetic,nobysame,backrefs]{amsrefs}

\begin{document}

\title[Small projective spaces and Stillman uniformity for sheaves]{Small projective spaces and\\ Stillman uniformity for sheaves}

\author{Daniel Erman}
\address{Department of Mathematics, University of Wisconsin, 480 Lincoln Drive, Madison, WI 53706}
\email{derman@math.wisc.edu}
\urladdr{\url{http://math.wisc.edu/~derman/}}

\author{Steven V Sam}
\address{Department of Mathematics, University of California, San Diego, 9500 Gilman Dr., La Jolla, CA 92093}
\email{ssam@ucsd.edu}
\urladdr{\url{http://math.ucsd.edu/~ssam/}}

\author{Andrew Snowden}
\address{Department of Mathematics, University of Michigan, 530 Church Street, Ann Arbor, MI 48109}
\email{asnowden@umich.edu}
\urladdr{\url{http://www-personal.umich.edu/~asnowden/}}


\thanks{DE was partially supported by NSF DMS-1601619. SS was partially supported by NSF DMS-1849173 and a Sloan Fellowship. AS was supported by NSF DMS-1453893.}

\begin{abstract}
We prove an analogue of Ananyan--Hochster's small subalgebra theorem in the context of sheaves on projective space, and deduce from this a version of Stillman's Conjecture for cohomology tables of sheaves. The main tools in the proof are Draisma's $\GL$-noetherianity theorem and the BGG correspondence.
\end{abstract}

\maketitle

\section{Introduction}

Ananyan and Hochster \cite{ananyan-hochster} recently established strong finiteness results on the Betti tables of graded modules, including Stillman's conjecture. In Boij--S\"oderberg theory, Betti tables of graded modules play a dual role to cohomology tables of coherent sheaves \cite{eisenbud-schreyer-jams,eisenbud-schreyer-icm}. It is therefore natural to look for analogs of the results of Ananyan--Hochster for cohomology tables. The purpose of this paper is to establish such results.

\subsection{The work of Ananyan--Hochster}

To put our results in context, we first review the work of Ananyan--Hochster. Fix a field $\bk$ and let $R$ be a finitely generated $\bk$-algebra with $R_0=\bk$ and $R_n=0$ for $n<0$. Let $M$ be a finitely generated graded $R$-module, generated in degrees $\geq d$. Define $\beta_{i,j}(M)=\dim_{\bk} \Tor_i^R(M, \bk)_j$, where the subscript here indicates the $j$th graded piece. These numbers are arranged into a table---the {\bf Betti table} of $M$---as follows:
\begin{displaymath}
\beta(M) = \begin{array}{|ccccccc} \hline
\beta_{0,d}& \beta_{1,d+1}&\beta_{2,d+2} &\beta_{3,d+3}&\beta_{4,d+4}&\cdots\\
\beta_{0,d+1}& \beta_{1,d+2}&\beta_{2,d+3} &\beta_{3,d+4}&\beta_{4,d+5}&\cdots\\
\vdots && \vdots&& \vdots \\
\beta_{0,n}& \beta_{1,n+1}&\beta_{2,n+2} &\beta_{3,n+3}&\beta_{4,n+4}&\cdots\\
\vdots && \vdots&& \vdots \\
\end{array}
\end{displaymath}
Some important properties of $M$ can be read off from the Betti table: the first two columns record the degrees of generators and relations; the Castelnuovo-Mumford regularity is encoded by the lowest row with a non-zero entry, that is by $\max\{j-i \mid \beta_{i,j}\ne 0\}$; and the projective dimension is the index of the last nonzero column. Moreover, the graded Betti numbers provide a refinement of the Hilbert polynomial of $M$.

Betti tables are invariant under flat extension of scalars: that is, if $R \to S$ is a flat morphism of $\bk$-algebras (of the sort under consideration) and $M$ is an $R$-module then $\beta(M)=\beta(S \otimes_R M)$. Therefore, to control $\beta(M)$ one could try to descend $M$ to a more manageable ring. This kind of descent is exactly what Ananyan--Hochster achieve:

\begin{theorem}[Ananyan--Hochster] \label{thm:ah}
Let $\bb$ and $\bb'$ be column vectors. Then there exist integers $n$ and $d$ with the following property. Suppose $S$ is a standard-graded polynomial algebra over $\bk$ (in however many variables) and $M$ is a graded $S$-module such that the first two columns of $\beta(M)$ are $\bb$ and $\bb'$. Then there exists a graded polynomial algebra $S_0$, generated by at most $n$ variables of degrees at most $d$, a flat $\bk$-algebra homomorphism $S_0 \to S$, and a graded $S_0$-module $M_0$ such that $M \cong S \otimes_{S_0} M_0$.
\end{theorem}
Over an algebraically closed field, this result follows by applying~\cite[Theorem~B]{ananyan-hochster} to the vector space spanned by the polynomials in a presentation matrix of $M$; see~\cite[\S6.6]{imperfect} for the case of an arbitrary field.  As suggested, the theorem allows one to control Betti tables, via standard noetherianity arguments over $S_0$.  Alternately, one could adapt~\cite[Theorem~5.14]{stillman}.  The precise statement is:

\begin{corollary}
Let $\bb$ and $\bb'$ be given. Then there exists a finite set $\cB=\cB(\bb,\bb')$ of tables such that if $S$ and $M$ are as in the theorem then $\beta(M) \in \cB$. In other words, given the first two columns of $\beta(M)$, there are only finitely many possibilities for the whole table.
\end{corollary}

Stillman's conjecture asserts that the projective dimension of a module over a standard-graded polynomial ring can be bounded from the first two columns of its Betti table. It therefore follows immediately from the corollary. We emphasize that the main point of the corollary and Stillman's conjecture is that the bounds are independent of the number of variables in the polynomial ring, and depend only on the number and degrees of generators and relations of the module.

\subsection{Main results}
We establish analogs of the above results in the setting of coherent sheaves. 
As noted above, this is a natural setting to seek such analogs, due to the duality in Boij--S\"oderberg theory~\cite{eisenbud-schreyer-jams, eisenbud-schreyer-icm,eisenbud-erman}.  (Koszul duality is another natural setting to ask about such analogs, but McCullough provided a negative result in that context~\cite{mccullough-exterior}.)

We begin with a definition. Let $X$ be a projective variety over $\bk$ equipped with an ample line bundle $\cO_X(1)$, and let $\cE$ be a coherent sheaf on $X$. Define $\gamma_{i,j}(\cE) = \dim_{\bk} \rH^i(X, \cE(j))$. These numbers are arranged into a table---the {\bf cohomology table} of $\cE$---as follows:
\begin{displaymath}
\gamma(\cE) = \begin{array}{ccccccc}
\hline
\cdots & \gamma_{n,-n-2}& \gamma_{n,-n-1}&\gamma_{n,-n} &\gamma_{n,1-n}&\gamma_{n,2-n}&\cdots\\
\cdots & \gamma_{n-1,-n-1}& \gamma_{n-1,-n}&\gamma_{n-1,1-n} &\gamma_{n-1,2-n}&\gamma_{n-1,3-n}&\cdots\\
& \vdots && \vdots&& \vdots \\
\cdots & \gamma_{1,-3}& \gamma_{1,-2}&\gamma_{1,-1} &\gamma_{1,0}&\gamma_{1,1}&\cdots\\
\cdots & \gamma_{0,-2}& \gamma_{0,-1}&\gamma_{0,0} &\gamma_{0,1}&\gamma_{0,2}&\cdots\\\hline
\end{array}
\end{displaymath}
Unlike Betti tables, cohomology tables extend infinitely to the left and right. The number of non-zero rows is bounded by the dimension of the support of $\cE$. We let $\gamma^k(\cE)$ denote the $k$th column of $\gamma(\cE)$, i.e., the sequence $(\gamma_{0,k}(\cE),\gamma_{1,k-1}(\cE),\dots, \gamma_{i,k-i}(\cE),\dots)$.

Cohomology tables are invariant under finite pushforwards: that is, if $(Y, \cO_Y(1))$ is a second projective variety and $i \colon X \to Y$ is a finite morphism such that $i^*(\cO_Y(1))=\cO_X(1)$ then $\gamma(i_*(\cE))=\gamma(\cE)$. Thus, to control the cohomology table of $\cE$, one could try to realize $\cE$ as a pushforward from a simpler variety. This is exactly our main result:

\begin{theorem} \label{mainthm}
Let $\bb$ and $\bb'$ be column vectors. Then there exists an integer $N=N(\bb,\bb')$ with the following property. Suppose that $\cE$ is a coherent sheaf on $\bP^n$, for some $n$, such that the zeroth and first columns of $\gamma(\cE)$ are $\bb$ and $\bb'$. Then there exists a linear embedding $i \colon \bP^{n_0} \to \bP^n$ with $n_0 \le N$ and a coherent sheaf $\cE_0$ on $\bP^{n_0}$ such that $\cE \cong i_*(\cE_0)$.
\end{theorem}

Note that the conclusion of this theorem is stronger than that of Theorem~\ref{thm:ah} as it realizes the sheaf as the pushforward under a \emph{linear} embedding. 
As before, this theorem allows us to control cohomology tables:

\begin{corollary} \label{cor:finite cohom tables}
Let $\bb$ and $\bb'$ be given. Then there exists a finite set $\Gamma=\Gamma(\bb,\bb')$ of tables such that if $\cE$ is as in the theorem then $\gamma(\cE) \in \Gamma$. In other words, given the zeroth and first columns of $\gamma(\cE)$, there are only finitely many possibilities for the whole table.
\end{corollary}

As a special case of this corollary, we obtain a Stillman-like conjecture for sheaves:

\begin{corollary} \label{cor:regbd}
Let $\bb$ and $\bb'$ be given. Then there exists an integer $r=r(\bb,\bb')$ with the following property: if $\cE$ is as in the theorem then the Castelnuovo--Mumford regularity of $\cE$ is at most $r$.
\end{corollary}

In fact, we prove the above corollary first, and use it in our proof of Theorem~\ref{mainthm}.

\begin{remark}
In Theorem~\ref{mainthm} we fixed the 0th and 1st columns of $\beta(\cE)$. In fact, one could fix any two consecutive columns. This follows from Theorem~\ref{mainthm} and the fact that the cohomology table of a twist $\cE(n)$ is obtained from that of $\cE$ by shifting all columns to the left by $n$. For Betti tables, however, one must use the first two columns (see Remark~\ref{rmk:Betti}).
\end{remark}

\subsection{Overview of proof}

The basic idea is to define an infinite dimensional parameter space for the sheaves $\cE$ under consideration, and use Draisma's $\GL$-noetherianity result~\cite{draisma} to control this space. We provide some more details.

First, a coherent sheaf $\cE$ on projective space can be encoded by its associated Tate resolution $\bT(\cE)$, which is a doubly infinite, everywhere exact complex over an exterior algebra.  This Tate resolution can, in turn, be encoded by any single differential.  This is reviewed in~\S\ref{sec:background}.

The initial data $\bb, \bb'$ determines the degrees of the entries of the 0th differential of the Tate resolution. We use this in \S\ref{ss:parameter} to define a parameter space for coherent sheaves with these columns in their cohomology table. This parameter space, which is infinite dimensional since we do not a priori fix the number of variables, is contained in a finite direct sum of exterior powers.

Inside of our parameter space, we define $Z_i$ to be the locus where the regularity of the corresponding sheaf is strictly bigger than $i$. In Corollary~\ref{cor:above zero row}, we show that $Z_i$ is Zariski closed.  This is one of the main technical points in the paper, and where we diverge from McCullough's counterexamples \cite{mccullough-exterior}.  If we were to study the analogous loci in the context of~\cite{mccullough-exterior} then the result would fail.  As McCullough explicitly shows, for a $1\times 1$ matrix consisting of a single quadratic form in an exterior algebra, the locus in $\bigwedge^2 \bk^\infty$ where the regularity of the cokernel has regularity at most $i$ is actually Zariski open, not closed.  The key fact that enables us to prove that $Z_i$ is Zariski closed in our case is that the matrices that arise in a Tate resolution are not generic, and so our ambient parameter space is a very specific sublocus inside of a direct sum of exterior powers; see Lemma~\ref{lem:reg bounds}.

We then have a chain of closed subsets $Z_1 \supseteq Z_2 \supseteq \cdots$. Draisma's noetherianity result~\cite{draisma} implies that $Z_r = Z_{r+1}$ for $r \gg 0$. Since any coherent sheaf on a projective space has finite regularity, we see that $Z_r=\emptyset$ for $r \gg 0$.  This gives a bound on regularity, yielding Corollary~\ref{cor:regbd}.

We then deduce Theorem~\ref{mainthm} from Corollary~\ref{cor:regbd}. For this, we use the fact that differentials beyond the regularity of a coherent sheaf in the Tate resolution are linear maps.  (This is shown in~\cite{efs}, and reviewed in \S\ref{sec:background}).  Using the regularity bound above, we can examine a fixed differential for all of our coherent sheaves.  Repeated applications of Corollary~\ref{cor:2nd column} (which also relies on Draisma's noetherianity result) show that the possible sizes of this differential come from a finite list, and hence there is a maximal number of linear forms used in this fixed differential.  This yields Theorem~\ref{mainthm}.

\subsection{Open questions}

Our work naturally suggested several further questions. First, can one prove Theorem~\ref{mainthm} via more elementary methods, e.g. without appealing to infinite dimensional noetherianity results? We have been able to handle some special cases directly, such as when the sheaf $\cE$ is the structure sheaf of a smooth subvariety, which suggests that this indeed may be possible. Second, could one prove Theorem~\ref{mainthm} via an ultraproduct argument, similar to what is done in ~\cite{stillman}?  Finally, it would be interesting to better understand how our result and that of Ananyan--Hochster interact with the duality between Betti tables and cohomology tables appearing in Boij--S\"oderberg theory.

\subsection{Outline}

In \S\ref{sec:background} we review the necessary material, with a specific emphasis on the BGG correspondence and on properties of Tate resolutions.  In \S\ref{sec:parameter}, we construct the parameter spaces $X^0_{\bd}$ which parametrize the coherent sheaves $\cE$ with $\gamma^k(\cE)=\bb$ and $\gamma^{k+1}(\cE)=\bb'$.  We also prove a number of key semicontinuity results about these spaces, including the fact that the space $Z_i$ mentioned above is Zariski closed. Finally, \S\ref{sec:proof} contains the proofs of Theorem~\ref{mainthm} and Corollary~\ref{cor:finite cohom tables}.

\subsection*{Acknowledgments}

We thank Jason McCullough for useful conversations.

\section{Background}\label{sec:background}

Throughout, we work over an arbitrary field $\bk$. 

\subsection{The BGG correspondence} \label{sec:BGG}
Here we give an overview of the key results of~\cite{efs} that will be relevant to us.  Let $V_n=\langle x_0,\dots,x_n\rangle$ be a graded vector space of dimension $n+1$ over $\bk$, where $\deg(x_i)=1$, and let $\Sym(V_n)$ be the symmetric algebra on $V_n$.  Let $W_n=\langle e_0,\dots,e_n\rangle$ be the dual graded vector space with corresponding dual basis, where $\deg(e_i)=-1$.  Let $\bigwedge W_n$ be the corresponding exterior algebra.  We often write $S=\Sym(V_n)$ or $E=\bigwedge W_n$ when the number of variables is fixed.

The BGG correspondence provides a correspondence between complexes of $S$-modules and complexes of $E$-modules.  It was introduced in~\cite{beilinson,bgg}, though the treatment developed in~\cite{efs} will be most relevant for our purposes.  There is a functor $\bR$ which transforms a graded $S$-module $M=\bigoplus_d M_d$ into a graded free linear complex $\bR(M)$ of $E$-modules, where $\bR(M)^i = M_i\otimes_{\bk} E(-i)$.  The differential $\partial^i \colon M_i\otimes_{\bk} E(-i)\to M_{i+1}\otimes_{\bk} E(-i-1)$ is determined by the image of the generators $M_i\to M_{i+1}\otimes W_n$ via the formula $\partial(m)=\sum_{i=0}^n (-1)^i x_im\otimes e_i$.  Alternately, one may view the map $M_i\to M_{i+1}\otimes W_n$ as the adjoint of the multiplication map $M_i\otimes V_n\to M_{i+1}$ given by the $S$-module structure on $M$.  There is a similar functor $\bL$ which transforms a graded $E$-module into a graded free linear complex of $S$-modules.  The functors $\bR$ and $\bL$ induce equivalences between the derived categories of $S$-modules and of $E$-modules.

\subsection{Tate resolutions}\label{subsec:tate}
Let $\cE$ be a coherent sheaf on $\bP^n=\Proj(\Sym(V_n))$. Work of Eisenbud, Fl{\o}ystad, and Schreyer shows that $\gamma(\cE)$ can be viewed as the Betti table of a minimal, free, doubly infinite, everywhere exact complex on $E=\bigwedge W_n$, known as a {\bf Tate resolution of $\cE$} and denoted $\bT(\cE)$~\cite[Theorem~4.1]{efs}.  To be more specific, if $\bT(\cE)=[\cdots \to \bT(\cE)^k \to \bT(\cE)^{k+1}\to \cdots]$ is a Tate resolution for $\cE$, then in cohomological degree $k$ we have the free module
\begin{equation}\label{eqn:Tate}
\bT(\cE)^k 
= \bigoplus_{i+j=k} E(-j)^{\gamma_{i,j}(\cE)}.
\end{equation}
The graded Betti numbers of the free module $\bT(\cE)^k$ are thus given by the entries of the cohomology table in a single column. 

\begin{example}\label{ex:cubic tate}
Let $C\subseteq \bP^2$ be a cubic curve. The cohomology table of its structure sheaf is
\begin{displaymath}
\gamma(\cO_{C}) = \begin{array}{c|ccccccc}
&& \gamma^{-2}(\cO_{C})&\gamma^{-1}(\cO_{C})&\gamma^0(\cO_{C}) &\gamma^1(\cO_{C}) &\gamma^2(\cO_{C}) &\\ \hline
2&\cdots &0&0&0&0&0&\cdots\\
1&\cdots &9&6&3&1&0&\cdots\\
0&\cdots &0&0&1&3&6&\cdots \\ \hline
\end{array}.
\end{displaymath}
A Tate resolution $\bT(\cO_{C})$ for $\cO_C$ has the form
\[
\cdots \overset{\partial^{-3}}{\longrightarrow} E(3)^9 \overset{\partial^{-2}}{\longrightarrow} E(2)^6 \overset{\partial^{-1}}{\longrightarrow} E(1)^3\oplus E^1 \overset{\partial^{0}}{\longrightarrow} E^1\oplus E(-1)^3 \overset{\partial^{1}}{\longrightarrow} E(-2)^6\to \cdots. \qedhere
\]
\end{example}

By fixing the entries in columns $0$ and $1$ of $\bT(\cE)$ as in Theorem~\ref{mainthm}, we are fixing the degrees of the entries of a matrix representing the differential $\partial^0$ from $\bT(\cE)$.  McCullough's examples in~\cite[Theorem~4.1]{mccullough-exterior} show that this alone is not sufficient to obtain finiteness results about the cokernel of a matrix over $E$. In particular, he shows that for each $\ell\geq 1$, the ideal generated by the single quadric
$e_1 \wedge e_2 + e_3 \wedge e_4+\dots+e_{2\ell-1} \wedge e_{2\ell}$
has regularity $\ell-2$.  Thus, even for a $1\times 1$ matrix $\phi$ containing a single quadric in the exterior algebra, there is no way to bound the regularity of $\coker \phi$ solely in terms of the degrees of the entries of $\phi$ (i.e., without reference to the number of variables).

But the matrices that arise in Tate resolutions turn out to be far from generic. To describe this, we let $\partial^k \colon \bT(\cE)^k\to \bT(\cE)^{k+1}$ be one of the differentials in a Tate resolution, and let $\phi^\vee$ be a matrix representing the dual of $\partial^k$, where the duality functor is $\Hom_E(-,E)$.  As the following lemma shows, it is always the case $\coker \phi^\vee$ has regularity $-k$, where we mean regularity as a graded $E$-module.  (This should not be confused with the Castelnuovo--Mumford regularity of the sheaf $\cE$.)  In other words, the Betti table of $\coker \phi^\vee$ has a built-in floor, and so we avoid McCullough's counterexamples.  The lemma follows easily from \cite[Theorem 4.1]{efs}.  But it will play an important role, and so we give an independent statement and proof.

\begin{lemma}\label{lem:reg bounds}
Fix a coherent sheaf $\cE$ on $\bP^n$ and fix some integer $k$.  Let $\partial^k$ denote the differential $\bT(\cE)^k\to \bT(\cE)^{k+1}$ and let $\phi^\vee \colon (\bT(\cE)^{k+1})^* \to (\bT(\cE)^k)^*$ be  a matrix representing the dual of $\partial^k$.  The regularity of $\coker( \phi^\vee)$ as a graded $E$-module is $-k$.  

In particular, if $\bG$ is a minimal free resolution of $\coker ( \phi^\vee)$ then $\bG_{j}$ is generated entirely in degree $-j-k$ for $j\gg 0$. 
\end{lemma}

\begin{proof}
It suffices to prove that $\bG_{j}$ is generated entirely in degree $-j-k$ for $j\gg 0$, as this will imply that the regularity of $\coker( \phi^\vee)$ is $-k$.  Now, since $\bT(\cE)$ is doubly exact and since $E$ is self-injective, we have that $\bG_{j} = (\bT(\cE)^{k+j})^*$.  It thus suffices to show that $\bT(\cE)^{k+j}$ is generated entirely in degree $k+j$ for fixed $k$ and $j\gg 0$.  This follows by combining \eqref{eqn:Tate} with Serre Vanishing for $\cE$. 
\end{proof}

\subsection{Cone-stable invariants}
We will consider numerical functions $\nu$ whose domain is the set of pairs $(E,M)$ where $E$ is an exterior algebra over $\bk$ generated by finitely many elements all of the same degree, and $M$ is a finitely generated graded $E$-module, and whose target is $\bZ \cup \{\infty\}$. Then $\nu$ is a {\bf module invariant} if $\nu(M)$ only depends on the pair $(E,M)$ up to isomorphism. Let $\bk[\eps]/ (\eps^2)$ be the exterior algebra on 1 generator. We say that a module invariant $\nu$ is {\bf cone-stable} if $\nu(M) = \nu(M \otimes_\bk \bk[\eps] / (\eps^2))$.

Finally, $\nu$ is {\bf weakly upper semi-continuous} if it is upper semi-continuous in flat families. More precisely, given a variety $V$ over $\bk$, let $\cM$ be a flat module over the sheaf of algebras $\cO_V \otimes_\bk E$. We require that the function $x\mapsto \nu(\cM_x)$ is upper semi-continuous, i.e., for all $n$, $\{x \mid \nu(\cM_x) \ge n\}$ is closed.

\section{Parameter spaces and semicontinuity results}\label{sec:parameter}
\subsection{Parameter spaces} \label{ss:parameter}
Since we are interested in the properties of cohomology tables on projective space,  without a priori fixing the number of variables, we will introduce certain limiting topological spaces.  This is similar to the setup in~\cite{genstillman}.

We fix for this section two column vectors $\bb=(\bb_0,\bb_1,\dots)$ and $\bb'=(\bb'_0,\bb'_1,\dots)$, and let $s$ be the maximal index such that $\bb_s$ or $\bb'_s$ is non-zero. To begin, we consider $E=\bigwedge W_n$ for a fixed $n$. The two vectors determine free modules $F=\bigoplus_{i \ge 0} E(i)^{\bb_i}$ and $F' = \bigoplus_{i \ge 0} E(i-1)^{\bb'_i}$. A homogeneous map $F \to F'$ is given by a matrix $\phi$ with entries of specified degrees.  We let $\bd_i$ represent the number of entries of degree $-i$ that could appear in such a matrix, and we write $\bd = (\bd_1,\bd_2,\dots)$.  Note that $\bd$ has only finitely many nonzero entries.  Since the matrices we are interested will arise from minimal free resolutions, there will be no nonzero entries of degree $0$, and thus we do not include a $\bd_0$.  Let $X_{\bd,n}=(\bigwedge^{1}W_n)^{\oplus \bd_1} \oplus (\bigwedge^{2}W_n)^{\oplus \bd_2}\oplus \cdots$ be the corresponding parameter space.  We will identify points $x \in X_{\bd,n}$ with matrices $\phi \colon F\to F'$ and sometimes abuse notation by writing $\phi \in X_{\bd,n}$.  We refer to such a matrix $\phi \colon F\to F'$ as a {\bf matrix of type $(\bb,\bb')$}.  An example will help clarify this setup.  

\begin{example}\label{ex:block matrix}
Following Example~\ref{ex:cubic tate}, let $\bb=(1,3,0,\dots)$ and $\bb'=(3,1,0,\dots)$.  Then, with the notation of the previous paragraph, $F=E^1\oplus E(1)^3$ and $F'=E(-1)^3\oplus E^1$.  A homogeneous map $F\to F'$ is thus determined by a block $4\times 4$ matrix which consists of $6$ entries of degree $-1$ (corresponding to submatrices $E^1\to E(-1)^3$ and $E(1)^3\to E^1$) and $9$ entries of degree $-2$ (corresponding to $E(1)^3\to E(-1)^3$).  So we set $\bd=(6,9,0,0,\dots)$.  Based only on degree considerations, this $4\times 4$ matrix also has a degree zero entry.  However, the matrices that arise in Tate resolutions never have units, and since we are parametrizing such matrices, we set all such entries to be zero.  A matrix of type $(\bb,\bb')$ thus corresponds to a $4\times 4$ matrix of the form:
\[
\begin{pmatrix}
a_1&b_{1,1}&b_{1,2}&b_{1,3}\\
a_2&b_{2,1}&b_{2,2}&b_{2,3}\\
a_3&b_{3,1}&b_{3,2}&b_{3,3}\\
0&a_{4}&a_{5}&a_6\\
\end{pmatrix},
\]
where the $a_i$ are elements in $E$ of degree $-1$ and the $b_{j,k}$ are elements in $E$ of degree $-2$.
\end{example}

Now we take a limit as $n\to \infty$. Let $W_{\infty}$ be the direct limit of the spaces $W_n$, under the standard inclusions, and let $E_{\infty}=\bigwedge{W_{\infty}}$. The inclusions $W_n \to W_{n+1}$ also induce inclusions $X_{\bd,n}\to X_{\bd,n+1}$, and we let $X_{\bd}$ be the direct limit, equipped with the direct limit topology.  Recall that a topological space $Z$ with the action of a group $G$ is said to be {\bf $G$-noetherian} if every descending chain of $G$-invariant closed subspaces of $Z$ stabilizes.  Since each space $X_{\bd,n}$ has a natural action of $\GL_n$ on it, which is induced by the action of $\GL_n$ on $W_n$, we obtain an induced action of $\GL:=\cup_n \GL_n$ on $X_{\bd}$.  We have that $X_{\bd}$ is $\GL$-noetherian by~\cite{draisma} and~\cite[Corollary~2.8]{genstillman}. As above, we identify points $\phi \in X_{\bd}$ with matrices, in this case with entries in $E_{\infty}$.  We can thus view the space $X_{\bd}$ as parametrizing matrices of type $(\bb,\bb')$.

This is not yet the parameter space that we want to work with.  As noted in Lemma~\ref{lem:reg bounds}, the matrices that appear in a Tate resolution with specified columns will have built-in regularity bounds.  We thus want to pass to the sublocus of $X_{\bd}$ of matrices satisfying that same regularity bound, as this will allow us to parametrize coherent sheaves with specified columns in its cohomology table.  

We will say that a coherent sheaf $\cE$ on $\bP^n$ has {\bf type $(\bb,\bb')$} if $\gamma^0(\cE)=\bb$ and $\gamma^1(\cE)=\bb'$.  If some $\phi\in X_{\bd}$ comes from a Tate resolution of a type $(\bb,\bb')$ sheaf $\cE$, then by Lemma~\ref{lem:reg bounds}, it must be the case that the regularity of $\coker \phi^\vee$ is zero.  So we let
\[
  X_{\bd}^0=\{\phi \in X_\bd \mid \reg \coker \phi^\vee =0\}
\]
be this subset.  We endow $X_{\bd}^0$ with the subspace topology, and thus $X_{\bd}^0$ is also $\GL$-noetherian.  A Tate resolution for a coherent sheaf $\cE$ of type $(\bb,\bb')$ thus induces an element $\phi\in X_{\bd}^0$.  

This construction is reversible as well.  Assume that we are given some $x\in X_{\bd}^0$ and choose $n$ minimally so that $x \in X_{\bd,n}$.  Let $\phi$ be the associated matrix of type $(\bb,\bb')$. Then, by Lemma~\ref{lem:reg bounds}, if $\bF$ is the minimal free resolution of $\coker \phi^\vee$, then $\bF_j$ is generated in degree $-j$ for $j\geq k$ for some $k$.  Writing $\bF^\vee_{\geq k}$ for the dual complex, we let $M=\bL(\bF^\vee_{\geq k})$ be the corresponding $\Sym(V_n)$-module, noting that, since $\bF^\vee_{\geq k}$ is a linear complex, applying $\bL$ yields a $\Sym(V_n)$-module, thought of as a  complex of modules concentrated in a single homological degree~\cite[Introduction]{efs}.  Let $\cE$ be the coherent sheaf on $\bP^n$ corresponding to $M$.  By construction, the truncated Tate resolution $\bT(\cE)^{\geq k}$ is isomorphic to $\bF^\vee_{\geq k}$ (see \S\ref{subsec:tate}) and thus
$\cE$ has type $(\bb,\bb')$.  In particular, by uniqueness of minimal free resolutions, there is a Tate resolution for $\cE$ where $\phi$ is the $k$th differential.  Since $\cE$ is determined by the point $x\in X_{\bd,n}$, we write $\cE_{x}$ for the sheaf on $\bP^n$ that results from this construction.  We will see in Lemma~\ref{lem:pushforward} that we would obtain an isomorphic sheaf for any $m\geq n$, and thus any sheaf cohomology groups of the resulting sheaf $\cE_x$ are cone-stable (see \S\ref{sec:background} for the definition of cone-stability and Remark~\ref{rmk:pushforward} for more).

\begin{remark}\label{rmk:dual matrix}
One potentially confusing point in this construction is that when associating a matrix $\phi$ to a point $x$, the matrix $\phi$ represents the differential $\partial^0\colon \bT^0(\cE_x)\to \bT^1(\cE_x)$, but the regularity bound from Lemma~\ref{lem:reg bounds} applies not to $\phi$ itself but $\phi^\vee$.  
In fact, there is no a priori bound on the regularity of $\coker \phi$.  However, Theorem~\ref{mainthm} implies the existence of such a bound.
\end{remark}

\subsection{Universal modules}\label{subsec:universal module}

Let $\cF=\bigoplus_{i \ge 0} E_{\infty}(i)^{\bb_i}$ and $\cF'= \bigoplus_{i \ge 0} E_{\infty}(i-1)^{\bb'_i}$ be the free $E_{\infty}$-modules corresponding to $\bb$ and $\bb'$.  There is a map $\Phi$ of $E_{\infty}$-modules $\Phi \colon \cF\to \cF'$ which is a universal matrix of type $(\bb,\bb')$ in the sense that by specializing the coefficients of $\Phi$, we can obtain any matrix $\phi \in X_{\bd}$ of type $(\bb,\bb')$. Set $\cM = \coker(\Phi^\vee)$ to be the universal module. In particular, if $x$ is a point of $X_{\bd}$ which corresponds to the matrix $\phi\colon F\to F'$, then the specialization $\Phi|_{x}$ is precisely $\phi$. Thus $\cM_x \cong \coker(\phi^\vee)$. These constructions also make sense at each finite level $n$, but when the context is clear, we will avoid subscripts to keep the notation manageable.

\begin{lemma}\label{lem:open}
  Let $U_{\bd}\subseteq X_{\bd}$ be the locus of points $x$ where the matrix $\phi^\vee$ corresponding to $x$ is a minimal presentation of $\cM_x$.
  \begin{enumerate}[\rm \indent (a)]
  \item  $U_{\bd}$ is Zariski open in $X_{\bd}$.  
  \item  For any point $x\notin U_{\bd}$, there exists a matrix $\psi$ of type $(\bb,\bb'')$ with $\bb''\leq \bb'$ in the termwise partial order, such that $\coker \psi^\vee \cong \cM_x$. 
  \end{enumerate}
\end{lemma}

\begin{proof}
Let $\phi\in X_{\bd}$ be a matrix such that $\phi^\vee$ is not a minimal presentation of its cokernel.  Since we have built into our construction that any such matrix has no entries that are units, this happens if and only if one of the columns of $\phi^\vee$ is in the image of the other columns.  Let $c$ be the number of columns of $\phi^\vee$; let $\phi^\vee_1,\dots,\phi^\vee_c$ be these columns and suppose that we can write $\phi^\vee_c = \sum_{i=1}^{c-1} f_i\phi^\vee_i$.  Let $\bd'$ be the degree sequence corresponding  to the entries of the matrix $(\phi^\vee_1 \phi^\vee_2\cdots \phi^\vee_{c-1})$ and let $\bd''$ be the degree sequence corresponding to the exterior polynomials $f_i$.  Then $\phi$ lies in the image of $X_{\bd'}\times X_{\bd''}\to X_{\bd}$.  In fact, this is equivalent to saying that $\phi^\vee_c$ is a polynomial combination of the other columns.

It follows that the complement of $U_\bd$ is a finite union of such images.  Part (b) is immediate from the construction, and $\cM_x$ is just the cokernel of $(\phi^\vee_1 \phi^\vee_2 \cdots \phi^\vee_{c-1})$. By \cite[Proposition 2.8]{genstillman}, to show that $U_\bd$ is open, it suffices to show that $U_{\bd,n}$ is open for all $n$. The map $X_{\bd',n} \times X_{\bd'',n} \to X_{\bd,n}$ is invariant under scaling, so we may projectivize both the source and target. In that case, the map becomes proper, so the image is closed. This means the image of the original map is an affine cone over a closed subset and hence is also closed. In particular, the complement of $U_{\bd,n}$ is a finite union of closed subsets, and hence $U_{\bd,n}$ is open.
\end{proof}

\subsection{Semicontinuity of Betti numbers}
Using the universal module $\cM$ we can obtain Zariski closed subloci of $X_{\bd,n}$ via any semi-continuous module invariant.  For instance:

\begin{lemma}\label{lem:hilbert values}
For any $n,r\in \bN$ and $e\in \bZ$, the locus of $x\in X_{\bd,n}$ such that $\dim (\cM_x)_e\geq r$ is Zariski closed.
\end{lemma}

\begin{proof}
As an $\cO_{X_{\bd,n}}$-module, $\cM$ splits as $\cM=\bigoplus_{d\in \bZ} \cM_d$.  We thus have $(\cM_x)_e \cong (\cM_e)_x$.  The statement now follows from standard semi-continuity results for coherent sheaves \cite[Example III.12.7.2]{hartshorne}.
\end{proof}

Since Hilbert function values such as $\dim(\cM_x)_e$ are not cone-stable, we will need to do a bit of extra work to obtain the required semi-continuity results on $X_{\bd}$.  For any finitely generated, graded $\bigwedge W_n$-module $N$, we define the invariant
\[
  \alpha_k(N) = \sum_{i=0}^\infty (-1)^i\beta_{i,k}(N).
\]
Note that for fixed $k$, $\beta_{i,k}(N)$ is nonzero for only finitely many $i$, so this is really a finite sum. Since Betti numbers are cone-stable, so are the invariants $\alpha_k$.  Since the notation in this subsection gets a bit heavy, we provide the following remark as an overview of the results.

\begin{remark}\label{rmk:summary}
We summarize the notation and results from this section with the following diagram. For any $x\in U_{\bd}$ (see Lemma~\ref{lem:open}) the Betti table of $\cM_x$ will look like (here $s$ is the maximal integer such that either $\bb_s$ or $\bb'_s$ is nonzero):
\addtocounter{equation}{-1}
\begin{subequations}
  \begin{equation}
    \label{eqn:beta mx}
\beta(\cM_x) = 
\begin{array}{c|c c c c c c c}&0&1&2 &\cdots&i&\cdots & s+2 \cr \hline
              s &\bb_s&\bb'_s&\beta_{2,s-2}&\cdots&\beta_{i,s-i} & \cdots & \beta_{s+2,-2}  \cr
              \vdots &\vdots&\vdots&\vdots&\cdots&\cdots&\cr
              1&\bb_1&\bb'_1&\beta_{2,-1}&\cdots&\beta_{i,1-i}& \cdots & \beta_{s+2,-s-1}  \cr
              0&\bb_0&\bb'_0&\beta_{2,-2}&\cdots&\beta_{i,-i}& \cdots & \beta_{s+2,-s-2} \cr
-1&-&-&\vdots&\ddots&&\ddots
       \end{array}. 
     \end{equation}
   \end{subequations}
If $x\notin U_{\bd}$, then at least one of the $\bb'_i$ can be reduced by one.  

There are no nonzero entries in row $s+1$ since this is the Betti table of a minimal free resolution. The invariants $\alpha_k(\cM_x)$ are obtained by taking the alternating sum of the entries lying on slope $1$ diagonals, starting with $\bb_k$.  The range $-2\leq k \leq s$ that appears in Lemma~\ref{lem:finite betti} and the related results ensure that we contain every diagonal up to and including the one passing through $\beta_{2,-2}$.  This plays a key role in the later application to sheaf cohomology tables, which have a built-in floor in the zeroth row (see Lemma~\ref{lem:reg bounds}), and the indexing is chosen to line up with that case.

Corollary~\ref{cor:2nd column} shows that for each $-2\leq j \leq s$ there are finitely many possibilities for the Betti numbers $\beta_{2,j}(\cM_x)$ as $x$ varies in $X_{\bd}$.
\end{remark}

\begin{lemma}\label{lem:hilb and alpha}
Let $N$ be a finitely generated, graded $\bigwedge W_n$-module.  
The Hilbert function of $N$ in degree $e$ is determined by $\{\alpha_k(N)\}_{k\geq e}$  via the formula
\addtocounter{equation}{-1}
\begin{subequations}
\begin{equation}\label{eqn:alpha and hilb}
\dim N_e = \sum_{j=e}^\infty \alpha_j(N) \binom{n+1}{j-e}.
\end{equation}
\end{subequations}
\end{lemma}

\begin{proof}
Let $\bF$ be a minimal free resolution of $N$.  Write $E=\bigwedge W_n$.  We have
\begin{align*}
\dim N_e &= \sum_{i=0}^\infty (-1)^i\dim (\bF_i)_e\\
\intertext{ where the sum is finite because $\bF$ is minimal.  We can rewrite this as}
\dim N_e&=\sum_{i=0}^\infty \sum_{j\in \bZ} (-1)^i\beta_{i,j}(N) \dim E(-j)_e\\
&=\sum_{j\in \bZ}\alpha_j(N) \dim E(-j)_e.\\
\intertext{However $E(-j)_e=0$ if $e>j$.  If $j\geq e$ then $\dim E(-j)_e = \binom{n+1}{j-e}$.  This yields:}
\dim N_e &=\sum_{j=e}^\infty \alpha_j(N) \binom{n+1}{j-e}. \qedhere
\end{align*}
\end{proof}

It is well-known that Betti numbers are weakly semi-continuous, i.e., they are semi-continuous in families with a fixed Hilbert function. In fact, an elementary argument shows that if one restricts to a family which has fixed Hilbert function values in degrees $>k$, then each Betti number $\beta_{i,j}$ with $j>k$ is semi-continuous.  The following lemma shows that one can do even a bit better than this if one only cares about Betti numbers $\beta_{i,j}$ with $i\geq 2$.  Note that, by definition of $\bb$, each Betti number $\beta_{i,j}(\cM_x)$ will be zero for $j>s$ by Remark~\ref{rmk:summary}. In particular, $\alpha_j(\cM_x)=0$ for $j>s$, so there is no need to fix these Hilbert function values.  

\begin{lemma}\label{lem:finite betti}
For some $-2\leq k \leq s$ fix a vector $\nu=(\nu_s,\nu_{s-1},\dots,\nu_{k+1})$.  Let
\[
  X^{\nu}_{\bd}=\{x \in X_{\bd} \mid \alpha_j(\cM_x) = \nu_j \text{ for all } j=k+1,\dots,s\}.
\]  
For any $i\geq 2$ and any $k\leq j\leq s$, the function $X^{\nu}_{\bd} \to \bN$ given by $x\mapsto \beta_{i,j}(\cM_x)$ is upper semi-continuous.
\end{lemma}

\begin{proof}
Since Betti numbers are cone-stable, it suffices to prove that the function is upper semi-continuous on $X^{\nu}_{\bd,n}$.   
By our assumption on $\nu$ and Lemma~\ref{lem:hilb and alpha}, for $x\in X^{\nu}_{\bd,n}$, the module $\cM_x$ has constant Hilbert function in degrees $k+1,\dots,s$.  It follows that $\cM_d$ is a finite rank vector bundle for $d=k+1,\dots,s$.  

By \cite[Theorem 7.8]{geom-syzygies}, $\beta_{i,j}(\cM_x)$ is the dimension of the homology of the complex:
\[
S_{i+1}W_n\otimes_{\bk} (\cM_x)_{j+i+1} \to S_iW_n\otimes_{\bk} (\cM_x)_{j+i}\to S_{i-1}W_n\otimes_{\bk} (\cM_x)_{j+i-1}.
\]
Note that, for $k\leq j \leq s$ and $i\geq 2$, all degrees in the above complex are in the range where $\cM$ is a vector bundle.  We also have $(\cM_{x})_d \cong (\cM_d)_x$ for all degrees $d$.  Thus, the locus of points $x$ where $\beta_{i,j}(\cM_x)\geq r$ is the locus where, after specializing to $x$, the homology of the following complex of finite rank vector bundles has dimension $\geq r$:
\[
S_{i+1}W_n\otimes_{\bk} \cM_{j+i+1} \to S_iW_n\otimes_{\bk} \cM_{j+i}\to S_{i-1}W_n\otimes_{\bk} \cM_{j+i-1}.
\]
Hence our desired result follows from standard upper-semicontinuity properties of coherent sheaves \cite[Example III.12.7.2]{hartshorne}.
\end{proof}

\begin{corollary}\label{cor:alpha}
For any $-2\leq k \leq s$, there are only finitely many distinct possibilities for $\alpha_k(\cM_x)$ for $x\in X_{\bd}$.
\end{corollary}

\begin{proof}
  Using Lemma~\ref{lem:open} it suffices to understand the distinct possibilities for $x \in U_{\bd'}$ for all $\bd' \le \bd$. For any given $\bd$, there are only finitely many possibilities for $\bd'$, so without loss of generality, it is enough to consider $x \in U_\bd$.

The case $k=s$ is immediate as $\alpha_s(\cM_x)=\bb_s$ by definition, and we proceed by descending induction on $k$.  We assume that there are finitely many distinct possibilities for $\alpha_j(\cM_x)$ for each $k+1\leq j \leq s$.  This implies that for $x \in U_{\bd}$ there are only finitely many possibilities for the vector $\nu = (\alpha_{k+1}(\cM_x),\dots,\alpha_s(\cM_x))$.  It thus suffices to prove the claim after restricting to one such $\nu$.

For a fixed $\nu$, we apply Lemma~\ref{lem:finite betti} to conclude, via $\GL$-noetherianity, that there are only finitely many possibilities for each $\beta_{i,k}(\cM_x)$ for each $i\geq 2$.  It follows that, even taking the union over all $\nu$, there are still only finitely many possibilities for each $\beta_{i,k}(\cM_x)$ for each $i\geq 2$.  We combine this with the following facts: $\beta_{0,k}=\bb_k$ and $\beta_{1,k} =\bb'_{k+1}$ (since $x \in U_{\bd}$); $\alpha_k(\cM_x)=\sum_{i=0}^\infty (-1)^i \beta_{i,k}(\cM_x)$ by definition; and $\beta_{i,k}(\cM_x)=0$ for $i>s-k$ by Remark~\ref{rmk:summary}. We conclude that there are only finitely many possibilities for $\alpha_k(\cM_x)$ as desired.
\end{proof}

Combining Lemma~\ref{lem:finite betti} and Corollary~\ref{cor:alpha} yields:

\begin{corollary}\label{cor:2nd column}
For any $-2\leq k \leq s$, there are only finitely many distinct possibilities for $\beta_{2,k}(\cM_x)$ for $x\in X_{\bd}$.
\end{corollary}

\begin{proof}
By Corollary~\ref{cor:alpha}, for $x\in X_{\bd}$ there are only finitely many possibilities for $\alpha_j(\cM_x)$ for each $-2 \leq j \leq s$.  We may thus restrict to a locus in $X_{\bd}$ where all of these values are constant.  Then we can apply Lemma~\ref{lem:finite betti} to conclude that $\beta_{2,k}(\cM_x)$ is upper semi-continuous and by $\GL$-noetherianity, this implies that there are only finitely many possible values.
\end{proof}

The following corollary will also be a useful special case.

\begin{corollary}\label{cor:above zero row}
For $i\geq 2$ we let $Z_i\subseteq X_{\bd}$ be the locus of $x$ such that $\beta_{i,j}(\cM_x)\ne 0$ for some $j>-i$.\footnote{In reference to \eqref{eqn:beta mx}, this means that $\beta(\cM_x)$ has a nonzero entry in column $i$ and some row strictly above the zeroth row.}  Then $Z_{i} \supseteq Z_{i+1}$ and each $Z_i$ is Zariski closed.
\end{corollary}

\begin{proof}
We have $Z_i\subseteq Z_{i-1}$ for degree reasons:  if $x\notin Z_{i-1}$ then the $(i-1)$st column of $\beta(\cM_x)$ has no nonzero entries above the zeroth row, and the same must be true for the $i$th column, and thus $x\notin Z_i$.

The locus $Z_2$ is a finite union of Zariski closed loci by Lemma~\ref{lem:finite betti}, and hence it is Zariski closed.  We will prove that $Z_k$ is also a finite union of closed loci for any $k>2$. Within $Z_2$, there are only finitely many distinct possibilities for the Betti numbers $\beta_{2,k}(\cM_x)$ with $-2\leq k \leq s$.  We restrict to one such possibility, which determines a new degree sequence, and we apply the same argument as we used for $Z_2$.  Iterating the argument yields the desired result for $Z_i$.  
\end{proof}

\section{Proof of Theorem~\ref{mainthm}}\label{sec:proof}
Tate resolutions are stable under pushforward via linear embeddings $\bP^n\to \bP^{n+1}$.  This is like a natural dual of the notion of cone-stability introduced in~\cite{genstillman}.
\begin{lemma}\label{lem:pushforward}
Let $\cE$ be a coherent sheaf on $\bP^n$ and let $\iota \colon \bP^n\to \bP^{n+1}$ a linear embedding. The Tate resolution $\bT(\cE)$ is stable under pushforward by $\iota$, namely: $\bT(\cE)\otimes_{\bigwedge W_n} \bigwedge W_{n+1}$ is isomorphic to $\bT(\iota_*\cE)$. 
\end{lemma}
\begin{proof}
We consider $V_n=\langle x_0,\dots,x_n\rangle$ as a direct summand of $V_{n+1}=\langle x_0,\dots,x_{n+1}\rangle=V_n\oplus V'$.  Since the Tate resolution is everywhere exact, and since minimal free resolutions over $\bigwedge W_{n+1}$ are unique up to isomorphism, it is enough to check the lemma for a single differential in the  complex.   Let $k\geq \reg(\cE)$.  Define the $\Sym(V_n)$-module $M=\bigoplus_{j\geq k} \rH^0(\bP^n, \cE(j))$ and the $\Sym(V_{n+1})$-module $M'=\bigoplus_{j\geq k} \rH^0(\bP^{n+1}, \iota_*\cE(j))$.  By definition, $M_k\cong M'_k$ for all $k$, and $x_{n+1}$ annihilates $M'$.

The differential $\rho^k \colon \bT(\iota_*\cE)^k\to \bT(\iota_*\cE)^{k+1}$ is defined by $M'_k\to M'_{k+1}\otimes W_{n+1}$ which is the adjoint of the map $M'_k\otimes V_{n+1}\to M'_{k+1}$.  However, since $x_{n+1}$ annihilates $M'$ we have:
\[
\xymatrix{
M'_k\otimes V_{n+1}\ar[r]\ar[d]& M'_{k+1}\\
M'_k\otimes V_n\ar[ru]
},
\]
where the vertical arrow is induced by the projection $V_{n+1}\to V_n$.  It follows that the image of $M'_k\to M'_{k+1}\otimes W_{n+1}$ lies in $M'_{k+1}\otimes W_n$.  In particular, $\rho^k$ only involves linear forms from $\bigwedge W_n$.

Now, the differential $\partial^k \colon \bT(\cE)^k\to \bT(\cE)^{k+1}$ is defined by $M_k\to M_{k+1}\otimes W_n$, which is the adjoint of the map $M_k\otimes V_n\to M_{k+1}$.  Since we have natural isomorphisms $M_k\cong M'_k$ and $M_{k+1}\cong M'_{k+1}$, we can identify the map $M_k\to M_{k+1}\otimes W_n$ with $M'_k\to M'_{k+1}\otimes W_n$.  This identifies the matrix $\phi$ with the matrix $\phi'$.  By exactness and uniqueness of Tate resolutions, this implies the desired statement.
\end{proof}

\begin{remark}\label{rmk:pushforward}
Let $x \in X^0_{\bd}\cap X_{\bd, n}$ and $\cE_{x}$ be the corresponding coherent sheaf on $\bP^n$.  Lemma~\ref{lem:pushforward} implies that if we then consider $x$ as a point $\phi \in X_{\bd}^0\cap X_{\bd,n+1}$, the corresponding sheaf would simply be $\iota_*(\cE_{x})$ where $\iota\colon \bP^n\to \bP^{n+1}$.  Since cohomology groups are invariant under closed immersions, it follows that the function which sends $x$ to the cohomology table of $\cE_x$ is well-defined on $X_{\bd}^0$.  We may henceforth refer to the cohomology table of $\cE_{x}$ without reference to $n$.
\end{remark}

\begin{lemma}\label{lem:lin Tate}
  Let $\cE$ be a coherent sheaf on $\bP^n$.  If there exists $m<n$ such that the matrices in the Tate resolution $\bT(\cE)$ only involve the variables $e_0,\dots,e_m$, then there exists a coherent sheaf $\cE'$ on $\bP^m$ and a linear embedding $\iota' \colon \bP^m\to \bP^n$ such that $\iota'_*\cE'\cong \cE$.
\end{lemma}

\begin{proof}
Let $V_m =\langle x_0,\dots,x_m\rangle$, considered as a split summand of $V_n$.  Let $k\geq \reg(\cE)$ and let $M=\bigoplus_{j\geq k} \rH^0(\bP^n,\cE(j))$.  For each $j\geq k$, the map $M_j\to M_{j+1}\otimes V_n^*$ factors through $M_j\to M_{j+1}\otimes V_m^*$.  Writing $V_n=V_m\oplus V''$,  with $V''=\langle x_{m+1},\dots,x_n\rangle$, the vector space on $x_{m+1},\dots,x_n$, this implies that for the multiplication map $M_j\otimes (V_m\oplus V'')\to M_{j+1}$, the map $M_j\otimes V''\to M_{j+1}$ is zero.  Since this holds for all $j\geq k$, this implies that $M$ is annihilated by $x_{m+1},\dots,x_n$.  In particular, $M$ can be viewed as a module over $\bk[x_0,\dots,x_m]$, and the statement follows.
\end{proof}

The following lemma is a special case of Theorem~\ref{mainthm} and Corollary~\ref{cor:regbd}, where $\bb_i$ and $\bb'_i$ are zero for $i>1$.  We will eventually use noetherian induction to reduce to this case.

\begin{lemma}\label{lem:linear}
Let $\bb=(\bb_0,0,0,\dots)$ and $\bb'=(\bb'_0,0,0,\dots)$ be column vectors.  For any coherent sheaf $\cE$ on $\bP^n$ with $\gamma^0(\cE)=\bb$ and $\gamma^1(\cE)=\bb'$ we have:
\begin{enumerate}[\indent \rm (a)]
\item $\cE$ is $0$-regular, and 
\item  If $n_0=\bb_0 \bb'_{0}-1$ then there exists a coherent sheaf $\cE'$ on $\bP^{n_0}$ and a linear map $\iota\colon \bP^{n_0}\to \bP^n$ such that $\cE\cong \iota_*\cE'$. 
\end{enumerate}
\end{lemma}

\begin{proof}
The fact that $\cE$ is $0$-regular follows from Remark~\ref{rmk:summary}: the shape of the Betti table implies that $\cE(j)$ does not have higher cohomology for $j\ge 0$. Choosing a Tate resolution $\bT(\cE)$ for $\cE$, we have $\bT(\cE)^0 = E^{\bb_0}$ and $\bT(\cE)^1 = E(-1)^{\bb'_0}$.  The matrix $\phi \colon E^{\bb_0}\to E(-1)^{\bb'_0}$ involves at most $\bb_0\bb'_1$ many distinct linear forms from $E$.  In particular, if we set $n_0=\bb_0\bb'_0-1$, then after a change of coordinates, we can assume that $\phi$ only involves the variables $e_0,\dots,e_{n_0}$.  Since $\bT(\cE)$ is doubly exact, it follows that the entire Tate resolution $\bT(\cE)$ can be defined over the subring in the variables $e_0,\dots,e_{n_0}$. The statement now follows from Lemma~\ref{lem:lin Tate}.
\end{proof}

\begin{proof}[Proof of Corollary~\ref{cor:regbd}]
For each $x\in X_{\bd}^0$ we write $\cE_x$ (as in Remark~\ref{rmk:pushforward}) for the corresponding coherent sheaf.  Choose $s$ so that $\bb_i=0$ and $\bb'_i=0$ for all $i>s$. The right side of the cohomology table of $\cE_x$ will look like:
\[
\gamma(\cE_x)= \begin{array}{c|cc c c c c c}&&\gamma^0 &\gamma^1&\gamma^2&\dots &\gamma^k&\cdots\cr\hline
              \vdots&\ddots&\ddots&\ddots&\ddots&\ddots&\ddots&\vdots \cr
              s+1&\cdots&0&0&0&\cdots&0&\cdots \cr
              s &\cdots&\bb_s&\bb'_s&\gamma_{s,2-s}&&\gamma_{s,k-s}&\cdots\  \cr
              \vdots& \cdots&\vdots&\vdots&\vdots&\ddots&\vdots&\cdots\cr
              1&\cdots&\bb_1&\bb'_1&\gamma_{1,1}&\cdots&\gamma_{1,k-1}&\cdots\  \cr
              0&\cdots&\bb_0&\bb'_0&\gamma_{0,2}&\cdots&\gamma_{0,k}&\cdots\ \cr \hline
              \end{array}        
\]
In particular, for $k\geq 0$, the column $\gamma^k$ can only have nonzero entries in rows $0,1,\dots,s$.

We can rephrase this in terms of the Betti table of $\cM_x$.  Namely, since $\cM_x$ has regularity zero by Lemma~\ref{lem:reg bounds}, the Betti table of $\cM_x$ will look like
\[
\beta(\cM_x) = 
\begin{array}{c|c c c c c c}&0&1&2 &\cdots&i&\cdots \cr \hline
              s &\bb_s&\bb'_s&\beta_{2,s-2}&\cdots&\beta_{i,s-i}&  \cr
              \vdots &\vdots&\vdots&\vdots&\cdots&\vdots& \cdots \\
              1&\bb_1&\bb'_1&\beta_{2,-1}&\cdots&\beta_{i,1-i}&  \cr
              0&\bb_0&\bb'_0&\beta_{2,-2}&\cdots&\beta_{i,-i}&\cr
       \end{array}. 
      \]
      where $\gamma_{i,j} = \beta_{i+j,-j}$.

For $i\geq 2$ we define $Z_i\subseteq X_{\bd}^0$ as in Corollary~\ref{cor:above zero row}, but restricted to $X_{\bd}^0$.  Namely: some $x\in X_{\bd}^0$ lies in $Z_i$ if and only if $\beta_{i,j}(\cM_x)\ne 0$ for some $j>-i$, which is in turn equivalent to the condition $\gamma_{k,i-k}(\cE_x)\ne 0$ for some $k>0$.  By Corollary~\ref{cor:above zero row}, the $Z_i$ form a descending chain of closed loci in $X_{\bd}^0$.  These thus eventually stabilize by $\GL$-noetherianity.  Since every element of $X_{\bd}^0$ has a cohomology table which is eventually entirely supported on the zeroth row (see Lemma~\ref{lem:reg bounds}), this must stabilize to the empty set. Thus, there is some $k_0$ such that each $\cE$ has regularity at most $k_0$.
\end{proof}

\begin{proof}[Proof of Theorem~\ref{mainthm}]
We have now shown that for every $x\in X_{\bd}^0$, the corresponding sheaf $\cE_{x}$ has a cohomology table that looks like:
\[
\gamma(\cE_x)=
\begin{array}{c|cc c c c c c c c}&\cdots&\gamma^0 &\gamma^1&\gamma^2&\dots &\gamma^{k_0-1}&\gamma^{k_0}&\gamma^{k_0+1}&\cdots\cr \hline
              s &\cdots&\bb_s&\bb'_s&\gamma_{s,2-s}& \cdots &\gamma_{s,k_0-1-s}&0&0&\cdots\  \cr
              \vdots &\ddots&\vdots&\vdots&\vdots& \ddots &\vdots&\vdots & \vdots & \cdots\cr
              1&\cdots&\bb_1&\bb'_1&\gamma_{1,1}&\cdots&\gamma_{1,k_0-2}&0&0&\cdots\  \cr
              0&\cdots&\bb_0&\bb'_0&\gamma_{0,2}&\cdots&\gamma_{0,k_0-1}&\gamma_{0,k_0}&\gamma_{0,k_0+1}&\cdots\ \cr\hline
              \end{array}.
\]
We apply Lemma~\ref{cor:2nd column} to conclude that there are only finitely many possibilities for $\gamma^2(\cE)$.  For each such possibility, we again apply Lemma~\ref{cor:2nd column}  to ensure that there are only finitely many possibilities for $\gamma^3(\cE)$.  Iterating this argument, we eventually conclude that there are only finitely many possibilities for $\gamma^{k_0}$ and $\gamma^{k_0+1}$.

Fix some such pair $(\gamma_{0,k_0},\gamma_{0,k_0+1})$.  Replacing $\cE$ by $\cE(k_0)$, we can apply Lemma~\ref{lem:linear} to see that $\cE$ is defined on a subprojective space of dimension $\gamma_{0,k_0}\gamma_{0,k_0+1}-1$.  We now cycle over all of the finitely many possibilities for this pair $(\gamma_{0,k_0},\gamma_{0,k_0+1})$ let $n_0$ be the maximal dimension of any of the resulting subprojective spaces.
\end{proof}

\begin{proof}[Proof of Corollary~\ref{cor:finite cohom tables}]
By Theorem~\ref{mainthm}, it suffices to prove the corollary for a fixed $n$.  
Since $X_{\bd,n}$ is reduced and noetherian, this follows from Corollary~\ref{cor:finite cohom noeth} below.
\end{proof}

\begin{lemma}\label{lem:cohom tables noetherian}
Let $V$ be a reduced noetherian scheme and let $\cF$ be a coherent sheaf on $V\times \bP^n$.  For a point $x\in V$ we write $\cF_x$ for the pullback to $\{x\} \times \bP^n$.
\begin{enumerate}[(a)]
	\item\label{noeth a} There is a nonempty open set $U\subseteq V$ and an integer $j_0$ such that: for all $x\in U, j\geq j_0$ and all $i=0,1,\dots,n$, the dimension of $\rH^i(\cF_x(j))$ does not depend on $x$.
	\item\label{noeth b} Write $\omega$ for $\cO_V\boxtimes \omega_{\bP^n}$ on $V\times \bP^n$.  For any $i$, there is a nonempty open subset $U\subseteq V$ such that $\cE xt^i(\cF,\omega)_x\cong\cE xt^i(\cF_x,\omega_{\bP^n})$ for all $x\in U$.
	\item\label{noeth c} There is a nonempty open set $U\subseteq V$ and integers $j_-$ and $j_+$ where $\gamma^j(\cF_x)$ does not depend on $x$ for all $x\in U$ and all $j\geq j_+$ or $j\leq j_-$.
	\item\label{noeth d} There is a nonempty open set $U\subseteq V$ where the entire cohomology table of $\cF_x$ does not depend on $x$ for all $x\in U$.
\end{enumerate}
\end{lemma}
\begin{proof}
Without loss of generality we can assume throughout that $V$ is integral.  

\ref{noeth a}:  By generic flatness we can assume that $\cF$ is flat over an open set $U$.  Since the regularity of a sheaf is semicontinuous in a flat family, we can moreover assume that $\cF_x$ is $j_0$-regular for all $x$ in some open set $U$.  It follows that the Hilbert polynomial of $\cF_x$ is constant for $x\in U$.  In addition, for each $j\geq j_0$ and $x\in U$, $\dim \rH^0(\cF_x(j))$ equals the value of that constant Hilbert polynomial evaluated at $j$ and the higher cohomology groups vanish.

\ref{noeth b}:  Each of the Ext-sheaves is computed as the $i$th homology of a certain complex of coherent sheaves on $V\times \bP^n$.  By generic flatness, we can find an open set $U$ where taking homology commutes with specialization to points $x\in U$ and the statement follows.

\ref{noeth c}:   Fix some $i$ between $0$ and $n$.  We assume that the conclusion of (b) holds and, by generic flatness, that $\cE xt^i(\cF,\omega)$ is flat over $U$.  We may further find some $j_i$ such that for all $x\in U$, $\cE xt^i(\cF_x,\omega_{\bP^n})(j)$ has no higher cohomology for $j\geq j_i$. We compute:
\[
\rH^0(\cE xt^i(\cF_x,\omega_{\bP^n})(j)) =\rH^0(\cE xt^i(\cF_x(-j),\omega_{\bP^n}))=\Ext^i(\cF_x(-j),\omega_{\bP^n})=\rH^{n-i}(\cF_x(-j)).
\]
Where the first equality is~\cite[Prop III.6.7]{hartshorne}, the second equality follows from the local-to-global  spectral sequence for Ext,
and the third equality is Serre duality.

By part \ref{noeth a}, perhaps after increasing $j_i$ or shrinking to a smaller open set, we can obtain that the dimension of $\rH^0(\cE xt^i(\cF_x,\omega_{\bP^n})(j))$ does not depend on $x$ for all $j\geq j_i$, and thus $\dim \rH^{n-i}(\cF_x(-j))$ does not depend on $x$ for all $j\geq j_i$.  It follows that, in sufficiently negative degrees, there is an open set $U$, where column $\gamma^j(\cF_x)$ does not depend on $x$  for all $x\in U$ and all $j\ll 0$.  A similar statement holds for $j\gg 0$ by part (a).

\ref{noeth d}:  Let $U$ be the open set from part \ref{noeth c}  and let $\eta\in U$ its generic point.  Since each individual entry of the cohomology table is semicontinuous, part \ref{noeth c}  shows that there are a finite number of Zariski open conditions which define the locus in $U$ where the cohomology of $\cF_x$ agrees with the cohomology table of $\cF_\eta$.
\end{proof}

\begin{corollary}\label{cor:finite cohom noeth}
Let $V$ be a reduced noetherian scheme and let $\cF$ be a coherent sheaf on $V\times \bP^n$.  Then
the fibers $\cF_x$ with $x\in V$ attain only finitely many distinct cohomology tables.
\end{corollary}
\begin{proof}
By Lemma~\ref{lem:cohom tables noetherian}\ref{noeth d}, there is a Zariski open set $U\subseteq V$ consisting of points with a constant cohomology table.  The statement then follows by noetherian induction.
\end{proof}

\begin{remark}\label{rmk:Betti}
Corollary~\ref{cor:finite cohom tables} is a finiteness result based on fixing {\em any} two consecutive columns of the cohomology table.  By contrast, similar results about Betti tables, like~\cite[Theorem~D]{ananyan-hochster}, require one to fix the first two columns of the Betti table.  

For Betti tables, this turns out to be necessary.  For instance, if $S=\bk[x,y]$ and let $I=(x^n,x^{n-1}y)$, then the Betti table of $S(n-1)/I$ is:
\[
\bordermatrix{ &&&\cr
              -n &1&-&- \cr
              \vdots &\vdots&\vdots&\vdots\cr
              0&-&2&1 \cr
              }            
\]
Thus, there are infinitely many different Betti table with $\beta_{1,1}=2$ and $\beta_{2,2}=1$ and all other entries in columns $1$ and $2$ equal to zero.

However, if we fix one additional piece of data---the degree in which the module is generated---then we can obtain a finiteness result.  This is a consequence of Stillman's Conjecture (which bounds the tail of such a Betti table) and Boij--S\"oderberg theory (which bounds the rest of the entries).  
\end{remark}

\end{document}